\documentclass[12pt]{amsart}
\usepackage{amsmath,amssymb,amsthm}
\newcommand{\re}{\mathbb{R}}
\newcommand{\co}{\mathbb{C}}

\newcommand{\cc}{\mathcal{C}}

\newcommand{\z}{\bar z}
\newcommand{\w}{\bar w}

\newcommand{\glb}{\mbox{\rm glb}}
\newcommand{\boxx}{\rule{2.12mm}{3.43mm}}

\title[Nonlinear inequalities and almost complex structures]{Some
  nonlinear differential inequalities and an application to H\"older
  continuous almost complex structures}

\date{\today}

\author{Adam Coffman}

\address{Department of Mathematical Sciences \\ Indiana University -
  Purdue University Fort Wayne \\ 2101 E.\ Coliseum Blvd. \\ Fort
  Wayne, IN, USA 46805-1499}

\email{CoffmanA@ipfw.edu}

\author{Yifei Pan}
\email{Pan@ipfw.edu}

\newtheorem{thm}{Theorem}[section]
\newtheorem{prop}[thm]{Proposition}
\newtheorem{lem}[thm]{Lemma}
\newtheorem{cor}[thm]{Corollary}

\theoremstyle{definition}

\newtheorem{defn}[thm]{Definition}
\newtheorem{notation}[thm]{Notation}
\newtheorem{example}[thm]{Example}

\theoremstyle{remark}

\newtheorem*{rem}{Remark}

\begin{document}

\begin{abstract}
  We consider some second order quasilinear partial differential
  inequalities for real valued functions on the unit ball and find
  conditions under which there is a lower bound for the supremum of
  nonnegative solutions that do not vanish at the origin.  As a
  consequence, for complex valued functions $f(z)$ satisfying
  $\partial f/\partial\bar z=|f|^\alpha$, $0<\alpha<1$, and
  $f(0)\ne0$, there is also a lower bound for $\sup|f|$ on the unit
  disk.  For each $\alpha$, we construct a manifold with an
  $\alpha$-H\"older continuous almost complex structure where the
  Kobayashi-Royden pseudonorm is not upper semicontinuous.
\end{abstract}

\subjclass[2000]{Primary 35R45; Secondary 32F45, 32Q60, 32Q65, 35B05}

\maketitle

\section{Introduction}\label{sec0}

We begin with an analysis of a second order quasilinear partial
differential inequality for real valued functions of $n$ real
variables,
\begin{equation}\label{eq15}
  \Delta u-B|u|^\varepsilon\ge0,
\end{equation}
where $B>0$ and $\varepsilon\in[0,1)$ are constants.  In Section
  \ref{sec1}, we use a Comparison Principle argument to show that
  (\ref{eq15}) has ``no small solutions,'' in the sense that there is
  a number $M>0$ such that any nonnegative solution $u$ on the unit
  ball which is nonzero at the origin must satisfy $u(\vec x)>M$ for
  some $\vec x$.

As an application of the results on the inequality (\ref{eq15}), we
show failure of upper semicontinuity of the Kobayashi-Royden
pseudonorm for a family of $4$-dimensional manifolds with almost
complex structures of regularity $\cc^{0,\alpha}$, $0<\alpha<1$.  This
generalizes the $\alpha=\frac12$ example of \cite{ipr}; it is known
(\cite{ir}) that the Kobayashi-Royden pseudonorm is upper
semicontinuous for almost complex structures with regularity
$\cc^{1,\alpha}$.

The construction of the almost complex manifolds in Section \ref{sec3}
is similar to that of \cite{ipr}.  One of the steps in \cite{ipr} is a
Maximum Principle argument applied to a complex valued function $h(z)$
satisfying the equation $\partial h/\partial\z=|h|^{1/2}$, to get the
property of no small solutions.  Our use of a Comparison Principle in
Section \ref{sec1} is different, and we arrive at this result:
\begin{thm}\label{thm1.1}
  For any $\alpha\in(0,1)$, suppose $h(z)$ is a continuous complex
  valued function on the closed unit disk, and on the set
  $\{z:|z|<1,h(z)\ne0\}$, $h$ has continuous partial derivatives and
  satisfies
  \begin{equation}\label{eq24}
    \frac{\partial
    h}{\partial\z}=|h|^\alpha.
  \end{equation}
   If $h(0)\ne0$ then $\sup|h|>S_\alpha$, where the constant
   $S_\alpha>0$ is defined by:
    \begin{equation}\label{eq21}
  S_\alpha=\left\{\begin{array}{cl}\left(\alpha(1-\alpha)\right)^{1/(2-2\alpha)}&\mbox{ if $0<\alpha\le\frac23$}\\
\ & \ \\
\left(\frac{4\alpha(1-\alpha)^2}{2-\alpha}\right)^{1/(2-2\alpha)}&\mbox{ if $\frac23\le\alpha<1$}\end{array}\right..
  \end{equation}
\end{thm}

Section \ref{sec1b} continues with an inequality related to (\ref{eq15}):
\begin{equation}\label{eq04}
  u\Delta u-B|u|^{1+\varepsilon}-C|\vec\nabla u|^2\ge0.
\end{equation}
For constants $B>0$, $C<1$ and $\varepsilon\le C$ (in particular, $C$
and $\varepsilon$ can be negative), Theorems \ref{lem0.3} and
\ref{lem0.2} show a similar property of no small solutions, using
elementary methods.

\section{Some differential inequalities}\label{sec1}

Let $D_R$ denote the open ball in $\re^n$ centered at $\vec0$, and let
$\overline{D}_R$ denote the closed ball.
\begin{lem}\label{lem1.2c}
  Given constants $B>0$ and $0\le\varepsilon<1$, let
  $$M=\left(\frac{B(1-\varepsilon)^2}{2(2\varepsilon+n(1-\varepsilon))}\right)^{\frac1{1-\varepsilon}}>0.$$
  Suppose the function $u:{\overline{D}}_1\to\re$ satisfies:
  \begin{itemize}
    \item $u$ is continuous on ${\overline{D}}_1$,
    \item $u(\vec x)\ge0$ for $\vec x\in D_1$,
    \item on the open set
  $\omega=\{\vec x\in D_1:u(\vec x)\ne0\}$, $u\in\cc^2(\omega)$,
    \item for $\vec x\in\omega$:
  \begin{eqnarray}
    \Delta u(\vec x)-B(u(\vec x))^{\varepsilon}\ge0.\label{eq01}
  \end{eqnarray}
  \end{itemize}
  If $u(\vec0)\ne0$, then $\displaystyle{\sup_{\vec x\in D_1}u(\vec x)>M}$.
\end{lem}
\begin{proof}
  Define a comparison function
  $$v(\vec x)=M|\vec x|^{\frac2{1-\varepsilon}},$$ so
  $v\in\cc^2(\re^n)$ since $0\le\varepsilon<1$.  By construction of
  $M$, it can be checked that $v$ is a solution of this nonlinear
  Poisson equation on the domain $\re^n$: $$\Delta v(\vec x)-B|v(\vec
  x)|^\varepsilon\equiv0.$$

  Suppose, toward a contradiction, that $u(\vec x)\le M$ for all $\vec
  x\in D_1$.  For a point $\vec x_0$ on the boundary of $\omega$,
  either $|\vec x_0|=1$, in which case by continuity, $u(\vec x_0)\le
  M=v(\vec x_0)$, or $0<|\vec x_0|<1$ and $u(\vec x_0)=0$, so $u(\vec
  x_0)\le v(\vec x_0)$.  Since $u\le v$ on the boundary of $\omega$,
  the Comparison Principle (\cite{gt} Theorem 10.1) applies to the
  subsolution $u$ and the solution $v$ on the domain $\omega$.  The
  relevant hypothesis for the Comparison Principle in this case is
  that the second term expression of (\ref{eq01}), $-BX^\varepsilon$,
  is weakly decreasing, which uses $B>0$ and $\varepsilon\ge0$.  (To
  satisfy this technical condition for all $X\in\re$, we define a
  function $c:\re\to\re$ by $c(X)=-BX^\varepsilon$ for $X\ge0$, and
  $c(X)=0$ for $X\le0$.  Then $c$ is weakly decreasing in $X$, $v$
  satisfies $\Delta v(\vec x)+c(v(\vec x))\equiv0$ and $u$ satisfies
  $\Delta u(\vec x)+c(u(\vec x))\ge0$.)

  The conclusion of the Comparison Principle is that $u\le v$ on
  $\omega$, however $\vec0\in\omega$ and $u(\vec0)>v(\vec0)$, a
  contradiction.
\end{proof}

Of course, the constant function $u\equiv0$ satisfies the inequality
(\ref{eq01}), and so does the radial comparison function $v$, so the
initial condition $u(\vec0)\ne0$ is necessary.

\begin{example}\label{ex2.2}
  In the $n=1$ case,
  $M=\left(\frac{B(1-\varepsilon)^2}{2(1+\varepsilon)}\right)^{\frac1{1-\varepsilon}}$.
  For points $c_1$, $c_2\in\re$, $c_1<c_2$, define a
  function $$u(x)=\left\{\begin{array}{cl}M(x-c_2)^{\frac2{1-\varepsilon}}&\mbox{if
    $x\ge c_2$}\\ \ & \ \\ 0&\mbox{if $c_1\le x\le c_2$}\\ \ &
  \ \\M(c_1-x)^{\frac2{1-\varepsilon}}&\mbox{if $x\le
    c_1$}\end{array}\right..$$ Then $u\in\cc^2(\re)$, and it is
  nonnegative and satisfies $u^{\prime\prime}=B|u|^\varepsilon$ (the
  case of equality in the $n=1$ version of (\ref{eq01})).  For
  $c_1<0<c_2$, this gives an infinite collection of solutions of the
  ODE $u^{\prime\prime}=B|u|^\varepsilon$ which are identically zero
  in a neighborhood of $0$, so the ODE does not have a unique
  continuation property.  For $c_1>0$ or $c_2<0$, the function $u$
  satisfies $u(0)\ne0$ and the other hypotheses of Lemma
  \ref{lem1.2c}, and its supremum on $(-1,1)$ exceeds $M$ even though
  it can be identically zero on an interval not containing $0$.
\end{example}
\begin{example}\label{ex2.3}
  In the case $n=2$, $B=1$, $\varepsilon=0$, (\ref{eq01}) becomes the
  linear inequality $\Delta u\ge1$ and the number $M=\frac14$ agrees
  with Lemma 2 of \cite{ipr}, which was proved there using a Maximum
  Principle argument.
\end{example}
The next Lemma shows how an inequality like (\ref{eq01}) with $n=2$
can arise from a first order PDE for a complex valued function.  By
introducing the parameter $\gamma$, the Proof is a generalization of a
calculation appearing in \cite{ipr}.  Let $z=x+iy$ be the coordinate
on $\co$.
\begin{lem}\label{lem1.1}
  Consider constants $\alpha$, $\gamma$ with $0<\alpha<1$ and
  $\gamma\ge\frac{2-\alpha}{2-2\alpha}$.  Let $\omega\subseteq\co$ be
  an open set, and suppose $h:\omega\to\co$ satisfies:
  \begin{itemize}
    \item $h\in\cc^1(\omega)$,
    \item $h(z)\ne0$ for all $z\in \omega$,
    \item $\displaystyle{\frac{\partial h}{\partial\z}=|h|^\alpha}$
      on $\omega$.
  \end{itemize}
  Then, the following inequality is satisfied on $\omega$:
  $$\Delta(|h|^{(1-\alpha)\gamma})\ge2\alpha(1-\alpha)\gamma|h|^{(1-\alpha)(\gamma-2)}.$$
\end{lem}
\begin{rem}
  The parameter $\gamma$ can be chosen arbitrarily large; to apply
  Lemma \ref{lem1.2c} to get the ``no small solutions'' result of
  Theorem \ref{thm1.1}, we need the RHS exponent
  $(1-\alpha)(\gamma-2)$ to be nonnegative, so $\gamma\ge2$.  In
  contrast, the case appearing in Lemma 1 of \cite{ipr} is
  $\alpha=\frac12$, $\gamma=\frac32$, so the exponent is $-\frac14$.
  Their approach to the ``no small solutions'' property (\cite{ipr}
  Theorem 2) is to use the negative exponent together with the result
  of Example \ref{ex2.3} to show that assuming $h$ has a small
  solution leads to a contradiction.  As claimed, such an argument can
  be generalized to apply to other nonpositive exponents, but
  $\frac{2-\alpha}{2-2\alpha}\le\gamma\le2$ holds only for
  $\alpha\le\frac23$.
\end{rem}
\begin{proof}[Proof of Lemma \ref{lem1.1}]
  We first want to show that $h$ is smooth on $\omega$, applying the
  regularity and bootstrapping technique of PDE to the equation
  $\partial h/\partial\z=|h|^\alpha$.  We recall the following fact
  (for a more general statement, see Theorem 15.6.2 of \cite{aim}):
  for a nonnegative integer $\ell$, and $0<\beta<1$, if
  $\varphi\in\cc^{\ell,\beta}_{loc}(\omega)$ and $u$ has first
  derivatives in $L^2_{loc}(\omega)$ and is a solution of $\partial
  u/\partial\z=\varphi$, then $u\in\cc^{\ell+1,\beta}_{loc}(\omega)$.
  In our case,
  $\varphi=|h|^\alpha\in\cc^1(\omega)\subseteq\cc^{0,\beta}_{loc}(\omega)$
  (since $h\in\cc^1(\omega)$ and is nonvanishing), and $u=h$ has
  continuous first derivatives, so we can conclude that
  $u=h\in\cc^{1,\beta}_{loc}(\omega)$.  Repeating gives that
  $h\in\cc^{2,\beta}_{loc}(\omega)$, etc.

  Since the conclusion is a local statement, it is enough to express
  $\omega$ as a union of simply connected open subsets $\omega_k$ and
  establish the conclusion on each subset.

  On the set $\omega_k$, there is a single-valued branch of $\log(h)$,
  so that the function
  $g(z)=(h(z))^{1-\alpha}=e^{(1-\alpha)\log(h(z))}$ is well-defined,
  smooth, and nonvanishing.
  \begin{eqnarray*}
    \frac{\partial g}{\partial\z}&=&(1-\alpha)h^{-\alpha}\frac{\partial h}{\partial\z}\\
    &=&(1-\alpha)h^{-\alpha}|h|^\alpha=(1-\alpha)h^{-\alpha}(h^{\alpha/2}\bar h^{\alpha/2})\\
    &=&(1-\alpha)\frac{\bar h^{\alpha/2}}{h^{\alpha/2}}=(1-\alpha)\left(\frac{\bar g}{g}\right)^{\frac{\alpha}{2(1-\alpha)}}.
   \end{eqnarray*}
  Let $g(z)=\rho e^{i\phi}$ be the polar form of $g$, for smooth real
  functions $\rho(z)>0$, $\phi(z)$.  Then the above equation turns into
  an equation which is first-order linear in $\rho$:
  \begin{eqnarray*}
    \frac{\partial}{\partial\z}(\rho e^
         {i\phi})&=&(1-\alpha)\left(e^{-2i\phi}\right)^{\frac{\alpha}{2(1-\alpha)}}\\ \implies\frac{\partial\rho}{\partial\z}+\rho
         i\frac{\partial\phi}{\partial\z}&=&(1-\alpha)e^{-i\frac{1}{1-\alpha}\phi}.
  \end{eqnarray*}
  The real and imaginary parts can be expressed in terms of the $x$
  and $y$ derivatives:
  \begin{eqnarray}
    \rho_x-\rho\phi_y&=&2(1-\alpha)\cos(\frac1{1-\alpha}\phi),\label{eq3}\\
    \rho_y+\rho\phi_x&=&-2(1-\alpha)\sin(\frac1{1-\alpha}\phi).\label{eq4}
  \end{eqnarray}
  The conclusion of the Lemma refers only to
  $|h|=\rho^{\frac1{1-\alpha}}$, so the remaining steps have the goal
  of eliminating $\phi$ from the system of equations.

  Multiplying (\ref{eq4}) by $\phi_x$ and (\ref{eq3}) by $\phi_y$ and
  subtracting gives:
  \begin{eqnarray}
    &&(\phi_x^2+\phi_y^2)\rho+\rho_y\phi_x-\rho_x\phi_y\nonumber\\
    &=&-2(1-\alpha)(\phi_y\cos(\frac1{1-\alpha}\phi)+\phi_x\sin(\frac1{1-\alpha}\phi)).\label{eq5}
  \end{eqnarray}
  The $x$ and $y$ derivatives of (\ref{eq3}) and (\ref{eq4}) are, respectively:
  \begin{eqnarray*}
    \rho_{xx}-\rho_x\phi_y-\rho\phi_{xy}&=&-2\sin(\frac1{1-\alpha}\phi)\phi_x,\\
    \rho_{yy}+\rho_y\phi_x+\rho\phi_{xy}&=&-2\cos(\frac1{1-\alpha}\phi)\phi_y,
  \end{eqnarray*}
  and adding gives a sum equal to a scalar multiple of the RHS of (\ref{eq5}):
  \begin{eqnarray}
    \Delta\rho+(\rho_y\phi_x-\rho_x\phi_y)&=&-2\left(\phi_y\cos(\frac1{1-\alpha}\phi)+\phi_x\sin(\frac1{1-\alpha}\phi)\right)\nonumber\\ &=&\frac1{1-\alpha}\left((\phi_x^2+\phi_y^2)\rho+\rho_y\phi_x-\rho_x\phi_y\right)\nonumber
  \end{eqnarray}
  \begin{equation}\label{eq6}
  \implies\Delta\rho-\frac1{1-\alpha}(\phi_x^2+\phi_y^2)\rho=\frac{\alpha}{1-\alpha}(\rho_y\phi_x-\rho_x\phi_y).
  \end{equation}
  The sum of squares of (\ref{eq3}) and (\ref{eq4}) is:
  \begin{eqnarray}
    4(1-\alpha)^2&=&(\rho_x-\rho\phi_y)^2+(\rho_y+\rho\phi_x)^2\nonumber\\
    &=&\rho_x^2+\rho_y^2+(\phi_x^2+\phi_y^2)\rho^2+2(\rho_y\phi_x-\rho_x\phi_y)\rho.\label{eq7}
  \end{eqnarray}
  Multiplying (\ref{eq6}) by $2\rho$ and combining with $(\ref{eq7})$ gives:
  \begin{eqnarray}
    &&2\rho\Delta\rho-\frac{2}{1-\alpha}(\phi_x^2+\phi_y^2)\rho^2\nonumber\\ &=&\frac{2\alpha}{1-\alpha}(\rho_y\phi_x-\rho_x\phi_y)\rho\nonumber\\ &=&\frac{\alpha}{1-\alpha}\left(4(1-\alpha)^2-(\rho_x^2+\rho_y^2+(\phi_x^2+\phi_y^2)\rho^2)\right)\nonumber\\ &\implies&2\rho\Delta\rho+\frac{\alpha}{1-\alpha}(\rho_x^2+\rho_y^2)\label{eq8}
    \\ &=&4\alpha(1-\alpha)+\frac{2-\alpha}{1-\alpha}(\phi_x^2+\phi_y^2)\rho^2.\nonumber
  \end{eqnarray}
  Considering the constant $\gamma>1$, the Laplacians satisfy:
  \begin{eqnarray*}
    \Delta(\rho^\gamma)&=&\gamma(\gamma-1)(\rho_x^2+\rho_y^2)\rho^{\gamma-2}+\gamma\rho^{\gamma-1}\Delta\rho\\ \implies\Delta\rho&=&\frac1{\gamma}\rho^{1-\gamma}\Delta(\rho^\gamma)-(\gamma-1)(\rho_x^2+\rho_y^2)\rho^{-1}.
  \end{eqnarray*}
  Substituting this expression into (\ref{eq8}) gives:
  \begin{eqnarray}
  &&\frac2{\gamma}\rho^{2-\gamma}\Delta(\rho^\gamma)\label{eq10}\\
  &=&4\alpha(1-\alpha)+\frac{2-\alpha}{1-\alpha}(\phi_x^2+\phi_y^2)\rho^2+\left(2(\gamma-1)-\frac{\alpha}{1-\alpha}\right)(\rho_x^2+\rho_y^2).\nonumber
  \end{eqnarray}
  By the hypotheses $0<\alpha<1$ and
  $\gamma\ge\frac{2-\alpha}{2-2\alpha}$, the coefficients
  $\frac{2-\alpha}{1-\alpha}$ and
  $2(\gamma-1)-\frac{\alpha}{1-\alpha}$ are nonnegative.  The
  conclusion is:
  \begin{eqnarray*}
    \frac2{\gamma}\rho^{2-\gamma}\Delta(\rho^\gamma)&\ge&4\alpha(1-\alpha)\\
    \implies\Delta(\rho^\gamma)&\ge&2\alpha(1-\alpha)\gamma\rho^{\gamma-2}\\
    \implies\Delta(|h^{1-\alpha}|^\gamma)&\ge&2\alpha(1-\alpha)\gamma|h^{1-\alpha}|^{\gamma-2}.
  \end{eqnarray*}
\end{proof}

\pagebreak

\begin{proof}[Proof of Theorem \ref{thm1.1}]
  Regarding the function $\rho=|h|^{1-\alpha}$ on the set $\omega$, a
  more precise conclusion from the hypotheses of Lemma \ref{lem1.1},
  which uses only $\gamma\ne0$, follows from (\ref{eq10}):
\begin{equation}\label{eq11}
  \frac2{\gamma}\rho^{2-\gamma}\Delta(\rho^\gamma)\ge4\alpha(1-\alpha)+\left(2(\gamma-1)-\frac{\alpha}{1-\alpha}\right)(\rho_x^2+\rho_y^2).
\end{equation}
Setting $\zeta(z)=\rho^\gamma=|h|^{(1-\alpha)\gamma}$, $\gamma>0$,
$\zeta$ is smooth on $\omega$ and (\ref{eq11}) implies the following
second order quasilinear differential inequality:
  \begin{equation}\label{eq12}
    \zeta\Delta\zeta\ge2\alpha(1-\alpha)\gamma|\zeta|^{1+(1-\frac2\gamma)}+\frac{2(\gamma-1)-\frac\alpha{1-\alpha}}{2\gamma}|\vec\nabla\zeta|^2.
  \end{equation}

In particular, if $h:{\overline{D}}_1\to\co$ is continuous, and on the
set $\omega=\{z\in D_1:h(z)\ne0\}$, $h\in\cc^1(\omega)$, then Lemma
\ref{lem1.2c} applies to (\ref{eq12}) for sufficiently large $\gamma$.
The hypotheses of Lemma \ref{lem1.2c} are satisfied with $n=2$,
$u=\zeta$, and $B=2\alpha(1-\alpha)\gamma>0$, when the second RHS term
of (\ref{eq12}) has a nonnegative coefficient
($\gamma\ge\frac{2-\alpha}{2-2\alpha}$) and the quantity
$\varepsilon=1-\frac2\gamma$ is in $[0,1)$ (for $\gamma\ge2$).  The
  conclusion of Lemma \ref{lem1.2c} is:
  \begin{eqnarray}
    \displaystyle{\sup_{z\in
        D_1}\zeta(z)}&>&M=\left(\frac14\cdot2\alpha(1-\alpha)\gamma(\frac2\gamma)^2\right)^{\gamma/2}\nonumber\\ \implies\displaystyle{\sup_{z\in
        D_1}|h(z)|}&>&\left(2\alpha(1-\alpha)\frac1\gamma\right)^{\frac1{2(1-\alpha)}}.\label{eq20}
  \end{eqnarray}
    We can choose
    $\gamma=\max\left\{2,\frac{2-\alpha}{2-2\alpha}\right\}$, so that
    the lower bound for the $\sup$ is $S_\alpha$ as appearing in
    (\ref{eq21}).
\end{proof}
Note that $S_\alpha\to0^+$ as $\alpha\to1^-$, and for
$\alpha=\frac23$, $S_{2/3}=\frac{2\sqrt{2}}{27}\approx0.10475656$.
This Theorem is used in the Proof of Theorem \ref{thm1}.

\begin{example}\label{ex2.5}
  As noted by \cite{ipr}, a one-dimensional analogue of Equation
  (\ref{eq24}) in Theorem \ref{thm1.1} is the well-known (for example,
  \cite{br} \S I.9) ODE $u^\prime(x)=B|u(x)|^\alpha$ for $0<\alpha<1$ and $B>0$,
  which can be solved explicitly.  By an elementary separation of
  variables calculation, the solution on an interval where $u\ne0$ is
  $|u(x)|=(\pm(1-\alpha)(Bx+C))^{\frac1{1-\alpha}}$.  The general
  solution on the domain $\re$ is, for
  $c_1<c_2$, $$u(x)=\left\{\begin{array}{cl}(1-\alpha)^{\frac1{1-\alpha}}(Bx-c_2)^{\frac1{1-\alpha}}&\mbox{if
    $x\ge c_2$}\\ \ & \ \\ 0&\mbox{if $c_1\le x\le c_2$}\\ \ &
  \ \\-(1-\alpha)^{\frac1{1-\alpha}}(c_1-Bx)^{\frac1{1-\alpha}}&\mbox{if
    $x\le c_1$}\end{array}\right..$$ So $u\in\cc^1(\re)$, and if
  $u(0)\ne0$, then
  $\displaystyle{\sup_{-1<x<1}|u(x)|>((1-\alpha)B)^{\frac1{1-\alpha}}}$.
\end{example}

\section{Lemmas for holomorphic maps}\label{sec2} 

We continue with the $D_R$ notation for the open disk in the complex
plane centered at the origin.  The following quantitative Lemmas on
inverses of holomorphic functions are used in a normal form step in
the Proof of Theorem \ref{thm1}.
\begin{lem}[\cite{garnett} Exercise I.1.]\label{prop1.2}
  Suppose $f:D_1\to D_1$ is holomorphic, with $f(0)=0$,
  $|f^\prime(0)|=\delta>0$.  For any $\eta\in(0,\delta)$, let
  $s=\left(\frac{\delta-\eta}{1-\eta\delta}\right)\eta$; then the
  restricted function $f:D_\eta\to D_1$ takes on each value $w\in D_s$
  exactly once.
\end{lem}
\begin{rem}
  The hypotheses imply $\delta\le1$ by the Schwarz Lemma.  We give a
  Proof in an Appendix, Section \ref{sec5}.
\end{rem}
\begin{lem}\label{lem1.2}
  Given $r>\frac{4\sqrt{2}}{3}$, if $Z_1:D_r\to D_2$ is holomorphic,
  with $Z_1(0)=0$, $Z_1^\prime(0)=1$, then there exists a continuous
  function $\phi:\overline{D}_1\to D_r$ which is holomorphic on $D_1$
  and which satisfies $(Z_1\circ\phi)(z)=z$ for all
  $z\in{\overline{D}}_1$.
\end{lem}
\begin{rem}
  The convenience of the constant $\frac{4\sqrt{2}}{3}\approx1.8856$,
  and the choice of $\eta=3r/8$ in the following Proof, are also
  explained in Section \ref{sec5}.  It follows from the Schwarz Lemma
  that $r\le2$, and it follows from the fact that $\phi$ is an inverse
  of $Z_1$ that $\phi(0)=0$ and $\phi^\prime(0)=1$.
\end{rem}
\begin{proof}
  Define a new holomorphic function $f:D_1\to D_1$ by
  \begin{equation}\label{eq23}
    f(z)=\frac12\cdot Z_1(r\cdot z),
  \end{equation}
  so $f(0)=0$, $f^\prime(0)=\frac r2$, and Lemma \ref{prop1.2} applies
  with $\delta=\frac r2$.  If we choose $\eta=\frac{3r}8$, then
  $s=\frac{3r^2}{64-12r^2}$, and the assumption
  $r>\frac{4\sqrt{2}}{3}$ implies $s>\frac12$.  It follows from Lemma
  \ref{prop1.2} that there exists a function $\psi:D_s\to D_{\eta}$
  such that $(f\circ\psi)(z)=z$ for all $z\in
  {\overline{D}}_{1/2}\subseteq D_s$; this inverse function $\psi$ is
  holomorphic on $D_{1/2}$.  The claimed function
  $\phi:{\overline{D}}_1\to D_{r\eta}\subseteq D_r$ is defined by
  $\phi(z)=r\cdot\psi(\frac12\cdot z)$, so for $z\in{\overline{D}}_1$,
  $$Z_1(\phi(z))=Z_1(r\cdot\psi(\frac12\cdot z))=2\cdot
  f(\psi(\frac12\cdot z))=2\cdot\frac12\cdot z=z.$$
\end{proof}

\pagebreak

\section{$J$-holomorphic disks}\label{sec3}

For $S>0$, consider the bidisk $\Omega_S=D_2\times D_S\subseteq\co^2$,
as an open subset of $\re^4$, with coordinates $\vec
x=(x_1,y_1,x_2,y_2)=(z_1,z_2)$ and the trivial tangent bundle
$T\Omega_S\subseteq T\re^4$.  Consider an almost complex structure $J$
on $\Omega_S$ given by a complex structure operator on $T_{\vec
  x}\Omega_S$ of the following form:
\begin{equation}\label{eq1}
  J(\vec
  x)=\left(\begin{array}{cccc}0&-1&0&0\\1&0&0&0\\0&\lambda&0&-1\\\lambda&0&1&0\end{array}\right),
\end{equation}
where $\lambda:\Omega_S\to\re$ is any function.

A map $Z:D_r\to\Omega_S$ is a $J$-holomorphic disk if
$Z\in\cc^1(D_r)$ and $dZ\circ J_{std}=J\circ dZ$, where $J_{std}$ is
the standard complex structure on $D_r\subseteq\co$.  For $J$ of the
form (\ref{eq1}), if $Z(z)$ is defined by complex valued component
functions,
\begin{equation}\label{eq2}
  Z:D_r\to\Omega_S: Z(z)=(Z_1(z),Z_2(z)),
\end{equation} then the $J$-holomorphic property implies that
$Z_1:D_r\to D_2$ is holomorphic in the standard way.

\begin{example}\label{ex1}
  If the function $\lambda(z_1,z_2)$ satisfies $\lambda(z_1,0)=0$ for
  all $z_1\in D_2$, then the map $Z:D_2\to\Omega_S:Z(z)=(z,0)$ is a
  $J$-holomorphic disk.
\end{example}

\begin{defn}\label{def1}
  The Kobayashi-Royden pseudonorm on $\Omega_S$ is a function
  $T\Omega_S\to\re:(\vec x,\vec v)\mapsto\|(\vec x,\vec v)\|_K$, defined
  on tangent vectors $\vec v\in T_{\vec x}\Omega_S$ to be the number
  $$\glb\left\{\frac1r:\exists\mbox{ a $J$-holomorphic }
  Z:D_r\to\Omega_S,\ Z(0)=\vec x,\ dZ(0)(\frac\partial{\partial x})=\vec
  v\right\}.$$
\end{defn}
Under the assumption that $\lambda\in\cc^{0,\alpha}(\Omega_S)$,
$0<\alpha<1$, it is shown by \cite{ir} and \cite{nw} that there is a
nonempty set of $J$-holomorphic disks through $\vec x$ with tangent
vector $\vec v$ as in the Definition, so the pseudonorm is a
well-defined function.

At this point we pick $\alpha\in(0,1)$ and set
$\lambda(z_1,z_2)=-2|z_2|^\alpha$.  Let $S=S_\alpha>0$ be the constant
defined by formula (\ref{eq21}) from Theorem \ref{thm1.1}.  Then,
$(\Omega_S,J)$ is an almost complex manifold with the following
property:
\begin{thm}\label{thm1}
  If $0\ne b\in D_S$ then $\|(0,b),(1,0)\|_K\ge\frac3{4\sqrt{2}}$.
\end{thm}
\begin{rem}
  Since $\frac3{4\sqrt{2}}\approx0.53$, and
  $\|(0,0),(1,0)\|_K\le\frac12$ by Example \ref{ex1}, the Theorem
  shows that the Kobayashi-Royden pseudonorm is not upper
  semicontinuous on $T\Omega_S$.
\end{rem}
\begin{proof}
  Consider a $J$-holomorphic map $Z:D_r\to\Omega_S$ of the form
  (\ref{eq2}), and suppose $Z(0)=(0,b)\in\Omega_S$ and
  $dZ(0)(\frac\partial{\partial x})=(1,0)$.  Then the holomorphic
  function $Z_1:D_r\to D_2$ satisfies $Z_1(0)=0$, $Z_1^\prime(0)=1$,
  and $Z_2\in\cc^1(D_r)$ satisfies $Z_2(0)=b$.

  Suppose, toward a contradiction, that there exists such a map $Z$
  with $b\ne0$ and $r>\frac{4\sqrt{2}}{3}$.  Then Lemma \ref{lem1.2}
  applies to $Z_1$: there is a re-parametrization $\phi$ which puts
  $Z$ into the following normal form: \begin{eqnarray*}
    (Z\circ\phi):{\overline{D}}_1&\to&\Omega_S\\ z&\mapsto&(Z_1(\phi(z)),Z_2(\phi(z)))=(z,f(z)),
 \end{eqnarray*}
 where $f=Z_2\circ\phi:{\overline{D}}_1\to D_S$ satisfies
 $f\in\cc^0({\overline{D}}_1)\cap\cc^1(D_1)$.  From the fact that
 $Z\circ\phi$ is $J$-holomorphic on $D_1$, it follows from the form
 (\ref{eq1}) of $J$ that if $f(z)=u(x,y)+iv(x,y)$, then $f$ satisfies
 this system of nonlinear Cauchy-Riemann equations on $D_1$:
 \begin{equation}\label{eq25}
    \frac{du}{dy}=-\frac{dv}{dx}\ \mbox{ and }
 \ \frac{du}{dx}+\lambda(z,f(z))=\frac{dv}{dy}
  \end{equation}
  with the initial conditions $f(0)=b$, $u_x(0)=u_y(0)=v_x(0)=0$ and
  $v_y(0)=\lambda(0,b)=-2|b|^\alpha$.  The system of equations implies
  \begin{eqnarray}
    \frac{\partial f}{\partial \z}&=&\frac12(\frac{\partial}{\partial x}(u+iv)+i\frac{\partial}{\partial y}(u+iv))\nonumber\\
    &=&\frac12(u_x-v_y+i(v_x+u_y))\nonumber\\
    &=&-\frac12\lambda(z,f(z))=|f|^\alpha.\label{eq9}
  \end{eqnarray}
  So, Theorem \ref{thm1.1} applies, with $f=h$.  The conclusion is that 
  \begin{eqnarray*}
    \displaystyle{\sup_{z\in
   D_1}|f(z)|}&>&S_\alpha,
  \end{eqnarray*}
  but this contradicts $|f(z)|<S=S_\alpha$.
\end{proof}
The previously mentioned existence theory for $J$-holomorphic disks
shows there are interesting solutions of the equation (\ref{eq9}), and
therefore also the inequality (\ref{eq12}).
\begin{example}\label{ex4.4}
  For $0<\alpha<1$, $(\Omega_S,J)$, $\lambda(z_1,z_2)=-2|z_2|^\alpha$
  as above, a map $Z:D_r\to\Omega_S$ of the form $Z(z)=(z,f(z))$ is
  $J$-holomorphic if $f\in\cc^1(D_r)$ and $f(x,y)=u(x,y)+iv(x,y)$ is a
  solution of (\ref{eq25}).  Again generalizing the $\alpha=\frac12$
  case of \cite{ipr}, examples of such solutions can be constructed
  (for small $r$) by assuming $v\equiv0$ and $u$ depends only on $x$,
  so (\ref{eq25}) becomes the ODE $u^\prime(x)-2|u(x)|^\alpha=0$.
  This is the equation from Example 2.5; we can conclude that
  $J$-holomorphic disks in $\Omega_S$ do not have a unique
  continuation property.
\end{example}

\section{Another differential inequality}\label{sec1b}

Here we consider another differential inequality, motivated by
\cite{ipr} and (\ref{eq12}).  The results of this Section do not play
a part in the construction in Section \ref{sec3}.

Unlike Lemma \ref{lem1.2c}, one of the hypotheses of the next Theorem
is that $u$ is strictly positive on the ball in $\re^n$.
\begin{notation}\label{not1}
  Let $\beta_n(r)=\int_{D_r}dV$ denote the volume of the ball
  $D_r\subseteq\re^n$.  Let $\sigma_n(r)=\int_{\partial D_r}dA$ denote the
  $(n-1)$-dimensional surface measure of the ball's boundary, $\partial D_r$.
  Define $$\kappa_n=\frac{\int_0^1\beta_n(r)dr}{\sigma_n(1)},$$ so
  $\kappa_1=\frac12$, $\kappa_2=\frac16$, $\kappa_3=\frac1{12}$,
  \ldots.
\end{notation}
\begin{thm}\label{lem0.3}
  Given constants $B>0$, $C<1$ and $\varepsilon\le C$,
  let $$M=\left((1-C)B\kappa_n\right)^{\frac1{1-\varepsilon}}>0.$$
  Suppose the function $u:D_1\to\re$
    satisfies:
  \begin{itemize}
    \item $u\in\cc^2(D_1)$,
    \item $u(\vec x)>0$ for $\vec x\in D_1$,
    \item for all $\vec x\in D_1$,
      \begin{eqnarray}
        u(\vec x)\Delta u(\vec x)&\ge&B|u(\vec x)|^{1+\varepsilon}+C|\vec\nabla
        u(\vec x)|^2.\label{eq03}
      \end{eqnarray}
  \end{itemize}
  Then, $\displaystyle{\sup_{\vec x\in D_1}u(\vec x)>M}$.
\end{thm}
\begin{proof}
  Using the assumption that $u(\vec x)>0$ and the identity
  \begin{eqnarray*}
    div(u^{-C}\vec\nabla u)&=&(u^{-C}u_{x_1})_{x_1}+\cdots+(u^{-C}u_{x_n})_{x_n}\\
    &=&u^{-C}\Delta u-Cu^{-C-1}|\vec\nabla u|^2,
  \end{eqnarray*}
  multiplying both sides of (\ref{eq03}) by $u^{-C-1}$ gives:
  \begin{eqnarray*}
    u^{-C}\Delta u&\ge&Bu^{\varepsilon-C}+C|\vec\nabla u|^2u^{-1-C}\\
    \implies div(u^{-C}\vec\nabla u)&\ge&Bu^{\varepsilon-C}.
  \end{eqnarray*}
  For $0\le r_1<1$, integrating over $D_{r_1}$,
  \begin{equation}\label{eq16}
    \int_{D_{r_1}}div(u^{-C}\vec\nabla u)dV\ge\int_{D_{r_1}}Bu^{\varepsilon-C}dV.
  \end{equation}
  We evaluate the LHS using a re-scaling, the Divergence Theorem, and the unit normal vector field $\vec\nu$ on the unit sphere $\partial D_1$.
  \begin{eqnarray}
    \int_{D_{r_1}}div((u(\vec x))^{-C}\vec\nabla u(\vec x))dV&=&\int_{D_1}div((u(r_1\vec x))^{-C}\vec\nabla u(r_1\vec x))r_1^ndV\nonumber\\
    &=&r_1^n\int_{\partial D_1}(u(r_1\vec x))^{-C}(\vec\nabla u(r_1\vec x))\cdot\vec\nu dA.\label{eq17}
  \end{eqnarray}
  For spherical coordinates $(r,\theta_1,\ldots,\theta_{n-1})=(r,\theta)$ on $\re^n$, (\ref{eq17}) can be written as
  $$r_1^n\int_{\partial D_1}(u(r_1,\theta))^{-C}(\vec\nabla
  u(r_1,\theta))\cdot\vec\nu
  dA=r_1^n\int_{\partial D_1}(u(r_1,\theta))^{-C}\frac{\partial u}{\partial
    r}(r_1,\theta)dA.$$ For $0\le r_1<r_2<1$, integrating
  $\int_{r_1=0}^{r_1=r_2}dr_1$ both sides of (\ref{eq16}) gives:
  \begin{eqnarray}
    0&\le&\int_{0}^{r_2}\int_{D_{r_1}}Bu^{\varepsilon-C}dVdr_1\nonumber\\
    &\le&\int_{0}^{r_2}r_1^n\int_{\partial D_1}(u(r_1,\theta))^{-C}\frac{\partial
    u}{\partial
    r}(r_1,\theta)dAdr_1\nonumber\\
    &\le&\int_{0}^{r_2}r_2^n\int_{\partial D_1}(u(r_1,\theta))^{-C}\frac{\partial
    u}{\partial
    r}(r_1,\theta)dAdr_1.\label{eq18}
  \end{eqnarray}
  The RHS can be re-arranged and estimated:
  \begin{eqnarray*}
    &&\int_{\partial
    D_1}\left(\int_{r_1=0}^{r_2}r_2^n(u(r_1,\theta))^{-C}\frac{\partial
    u}{\partial r}(r_1,\theta)dr_1\right)dA\\ &=&r_2\int_{\partial
    D_1}\left(\frac1{-C+1}(u(r_2,\theta))^{-C+1}-\frac1{-C+1}(u(\vec0))^{-C+1}\right)r_2^{n-1}dA\\
    &=&\frac{r_2}{1-C}\int_{\partial D_{r_2}}\left((u(\vec
    x))^{1-C}-(u(\vec0))^{1-C}\right)dA\\
    &\le&\frac{r_2}{1-C}\left(\left(\sup_{\vec x\in \partial
    D_{r_2}}(u(\vec
    x))^{1-C}\right)-(u(\vec0))^{1-C}\right)\sigma(r_2).
  \end{eqnarray*}
  Using $1-C>0$, the inequality (\ref{eq18}) implies:
  $$\sup_{\vec x\in \partial D_{r_2}}(u(\vec
  x))^{1-C}\ge(u(\vec0))^{1-C}+\frac{1-C}{r_2\sigma(r_2)}\int_{r_1=0}^{r_2}\int_{D_{r_1}}B(u(\vec
  x))^{\varepsilon-C}dVdr_1.$$ Suppose, toward a contradiction, that
  $\displaystyle{\sup_{\vec x\in D_1}u\le M}$.  Then, since
  $\varepsilon\le C$, $(u(\vec x))^{\varepsilon-C}\ge
  M^{\varepsilon-C}$ and
  \begin{eqnarray}
    \sup_{\vec x\in \partial D_{r_2}}(u(\vec
    x))^{1-C}&\ge&(u(\vec0))^{1-C}+\frac{1-C}{r_2\sigma(r_2)}\int_{0}^{r_2}\int_{D_{r_1}}BM^{\varepsilon-C}dVdr_1\nonumber\\ &=&(u(\vec0))^{1-C}+\frac{1-C}{r_2\sigma(r_2)}BM^{\varepsilon-C}\int_0^{r_2}\beta_n(r_1)dr_1.\label{eq19}
  \end{eqnarray}
  Since (\ref{eq19}) holds for all $r_2\in(0,1)$, 
  \begin{eqnarray*}
    \sup_{\vec x\in D_1}u(\vec x)&\ge&\lim_{r_2\to1^-}\sup_{\vec x\in
      \partial D_{r_2}}u(\vec
    x)\\ &\ge&\lim_{r_2\to1^-}\left((u(\vec0))^{1-C}+\frac{1-C}{r_2}BM^{\varepsilon-C}\frac{\int_0^{r_2}\beta_n(r_1)dr_1}{\sigma_n(r_2)}\right)^{\frac1{1-C}}\\ &>&\left((1-C)BM^{\varepsilon-C}\kappa_n\right)^{\frac1{1-C}}.
  \end{eqnarray*}
  The last quantity is exactly $M$ by construction, a contradiction.
\end{proof}
Theorem \ref{lem0.3} can be applied to the inequality (\ref{eq12}),
where the condition $\varepsilon\le C$ becomes
$1-\frac2\gamma\le\frac1{2\gamma}\left(2(\gamma-1)-\frac\alpha{1-\alpha}\right)$.
However, for $\alpha\in(0,1)$ and $\gamma>0$, the condition is
equivalent to $\alpha\le\frac23$.
\begin{cor}\label{cor3.2}
  Given constants $B>0$, $C<1$, $\varepsilon\le C$, and $M$ as in Theorem \ref{lem0.3}, suppose the function $u:D_1\to\re$
    satisfies:
  \begin{itemize}
    \item $u\in\cc^2(D_1)$,
    \item for all $\vec x\in D_1$, $u(\vec x)\Delta
      u(\vec x)\ge B|u(\vec x)|^{1+\varepsilon}+C|\vec\nabla u(\vec x)|^2$.
  \end{itemize}
  If $\displaystyle{\sup_{\vec x\in D_1}u(\vec x)\le M}$, then there is
  some $\vec x\in D_1$ with $u(\vec x)\le0$.  \boxx
\end{cor}

Functions satisfying a differential inequality of the form
(\ref{eq15}) or (\ref{eq04}) also satisfy a Strong Maximum Principle;
the only condition is $B>0$.

\begin{thm}\label{thm3.3}
  Given any open set $\Omega\subseteq\re^n$, and any constants $B>0$,
  $C,\varepsilon\in\re$, suppose the function $u:\Omega\to\re$ satisfies:
  \begin{itemize}
    \item $u$ is continuous on $\Omega$,
    \item on the set $\omega=\{\vec x\in\Omega:u(\vec x)>0\}$,
      $u\in\cc^2(\omega)$,
    \item on the set $\omega$, $u$ satisfies $$u\Delta u-B|u|^{1+\varepsilon}-C|\vec\nabla
  u|^2\ge0.$$
  \end{itemize}
  If $u(\vec x_0)>0$ for some $\vec x_0\in\Omega$, then $u$ does not
  attain a maximum value on $\Omega$.
\end{thm}
\begin{proof}
  Note that the constant function $u\equiv0$ is the only locally
  constant solution of the inequality for $B>0$.  If $B=0$ then some
  other constant functions would also be solutions.

  Given a function $u$ satisfying the hypotheses, $\omega$ is a
  nonempty open subset of $\Omega$.  Suppose, toward a contradiction,
  that there is some $\vec x_1\in\Omega$ with $u(\vec x)\le u(\vec
  x_1)$ for all $x\in\Omega$.  In particular, $u(\vec x_1)\ge u(\vec
  x_0)>0$, so $\vec x_1\in\omega$.  Let $\omega_1$ be the connected
  component of $\omega$ containing $\vec x_1$.

  For $\vec x\in\omega_1$, $u$ satisfies the linear, uniformly
  elliptic inequality $$\Delta u(\vec x)+(-B(u(\vec
  x))^{\varepsilon-1})u(\vec x)+(-C\frac{\vec\nabla u(\vec x)}{u(\vec
  x)})\cdot\vec\nabla u(\vec x)\ge0,$$ where the coefficients (defined
  in terms of the given $u$) are locally bounded functions of $\vec
  x$, and $(-B(u(\vec x))^{\varepsilon-1})$ is negative for all $\vec
  x\in\omega$.  It follows from the Strong Maximum Principle
  (\cite{gt} Theorem 3.5) that since $u$ attains a maximum value at
  $\vec x_1$, then $u$ is constant on $\omega_1$.  Since the only
  constant solution is $0$, it follows that $u(\vec x_1)=0$, a
  contradiction.
\end{proof}

In the $n=1$ case, we can get a result similar to Theorem
\ref{lem0.3}, but taking advantage of more information on initial
conditions.  We are interested in the ordinary differential inequality
$$uu^{\prime\prime}-B|u|^{1+\varepsilon}-C(u^\prime)^2\ge0.$$ By
assuming $u(0)>0$ and $u^\prime(0)\ge0$, we will show that the lower
bound $M$ for the supremum on $D_1=(-1,1)$ is exceeded on the
right-side subinterval $[0,1)$.  Analogously, if $u^\prime(0)\le0$,
  then $\sup u>M$ on the left-side interval $(-1,0]$.  The first Lemma
is a technical step for nonvanishing.

\begin{lem}\label{lem0.1}
  Given constants $p\in\re$, $B>0$, $C\ge-1$, suppose the function
  $u:(-1,1)\to\re$ satisfies:
  \begin{itemize}
    \item $u\in\cc^2([0,1))$,
    \item for $0<t<1$,
      \begin{equation}\label{eq02}
        u(t)u^{\prime\prime}(t)\ge B|u(t)|^p+C(u^\prime(t))^2.
      \end{equation}
  \end{itemize}
  If $u(0)>0$ and $u^\prime(0)\ge0$, then $u(t)>0$ on $[0,1)$.
\end{lem}
\begin{proof}
  Suppose, toward a contradiction, that $u$ attains some nonpositive
  value, so by continuity, there is some $x\in(0,1)$ with $u(x)=0$.
  Applying $\int_0^xdt$ to both sides of (\ref{eq02}), and integrating
  $uu^{\prime\prime}$ by parts gives:
  $$u(x)u^\prime(x)-u(0)u^\prime(0)-\int_0^x(u^\prime(t))^2dt\ge
  B\int_0^x|u(t)|^pdt+C\int_0^x(u^\prime(t))^2dt$$
  $$\implies 0\ge
  u(0)u^\prime(0)+B\int_0^x|u(t)|^pdt+(C+1)\int_0^x(u^\prime(t))^2dt.$$
  However, the first and third terms on RHS are $\ge0$ and the middle
  term is positive.
\end{proof}

In the following Theorem, the constant $M$ is the same as the bound
from Theorem \ref{lem0.3} for $n=1$.  The condition $-1\le C$ did not
appear in Theorem \ref{lem0.3}, it comes from Lemma \ref{lem0.1}.  The
condition $\varepsilon\le C$ means that the Theorem does not apply to
Example \ref{ex2.2}, where $C=0\le\varepsilon$, except for
$\varepsilon=0$.
\begin{thm}\label{lem0.2}
  Given constants $B>0$, $-1\le C<1$ and $\varepsilon\le C$,
  let $$M=\left(\frac12(1-C)B\right)^{\frac1{1-\varepsilon}}>0.$$
  Suppose the function $u:(-1,1)\to\re$ satisfies:
  \begin{itemize}
    \item $u\in\cc^2([0,1))$,
    \item for $0<t<1$,
      \begin{equation}\label{eq22}
        u(t)u^{\prime\prime}(t)\ge B|u(t)|^{1+\varepsilon}+C(u^\prime(t))^2.
      \end{equation}
  \end{itemize}
  If $u(0)>0$ and $u^\prime(0)\ge0$, then $\displaystyle{\sup_{0\le y<1}u(y)>M}$.
\end{thm}
\begin{proof}
  Lemma \ref{lem0.1} applies, so $u(t)>0$ on $[0,1)$, and we can
    multiply both sides of (\ref{eq22}) by $(u(t))^{-C-1}$ to get:
  $$(u(t))^{-C}u^{\prime\prime}(t)\ge
    B(u(t))^{\varepsilon-C}+C(u(t))^{-C-1}(u^\prime(t))^2.$$ Applying
    $\int_0^xdt$ to both sides, and integrating LHS by parts gives:
  \begin{eqnarray*}
    &&(u(x))^{-C}u^\prime(x)-(u(0))^{-C}u^\prime(0)-\int_0^x(-C)(u(t))^{-C-1}(u^\prime(t))^2dt\\
    &\ge&B\int_0^x(u(t))^{\varepsilon-C}dt+C\int_0^x(u(t))^{-C-1}(u^\prime(t))^2dt.
  \end{eqnarray*}
  Two integrals cancel exactly, and we can neglect the nonnegative
  constant term $(u(0))^{-C}u^\prime(0)$, to conclude:
  $$(u(x))^{-C}u^\prime(x)\ge B\int_0^x(u(t))^{\varepsilon-C}dt,$$ for
  all $x\in(0,1)$.  Then, applying $\int_0^ydx$ to both sides gives:
  \begin{equation*}
    \frac1{1-C}\left((u(y))^{1-C}-(u(0))^{1-C}\right)\ge
  B\int_{x=0}^{x=y}\left(\int_{t=0}^{t=x}(u(t))^{\varepsilon-C}dt\right)dx.
  \end{equation*}
  Suppose, toward a contradiction, that $\displaystyle{\sup_{0\le
      y<1}u(y)\le M}$, so $0<u(t)\le M$ and, since $\varepsilon-C\le0$,
  $(u(t))^{\varepsilon-C}\ge M^{\varepsilon-C}$.  It follows that
  $$B\int_{0}^{y}\int_{0}^{x}(u(t))^{\varepsilon-C}dtdx\ge
  B\int_{0}^{y}\int_{0}^{x}M^{\varepsilon-C}dtdx=BM^{\varepsilon-C}\frac12y^2.$$
  Since the constant term $u(0)$ is positive, and $\sup y^2=1$,
  \begin{eqnarray*}
    u(y)&\ge&\left((u(0))^{1-C}+(1-C)BM^{\varepsilon-C}\frac12y^2\right)^{\frac1{1-C}}\\
    \implies\sup_{0\le y<1}u(y)&\ge&\sup_{0\le y<1}\left((u(0))^{1-C}+(1-C)BM^{\varepsilon-C}\frac12y^2\right)^{\frac1{1-C}}\\
    &>&\left((1-C)BM^{\varepsilon-C}\frac12\right)^{\frac1{1-C}}.
  \end{eqnarray*}
  The last quantity is exactly $M$ by construction, a contradiction.
\end{proof}

\section{Appendix: Solution of Garnett's exercise}\label{sec5}

We will be using both the Euclidean distance $|z-w|$ in $\co$ and the
pseudohyperbolic distance $\rho_H(z,w)=\left|\frac{z-w}{1-\w
  z}\right|$ for $|z|$, $|w|<1$.  We follow the notation of
(\cite{garnett} \S I.1) for the disks:
\begin{notation}
  For $r>0$ and $z_0\in\co$, let $D(z_0,r)$ denote the Euclidean disk
  with center $z_0$ and radius $r$, so $D(0,r)=D_r$ is the special case with
  $z_0=0$.  For $0<r<1$ and $z_0\in D_1$,
  denote $$K(z_0,r)=\{z\in D_1:\rho_H(z,z_0)<r\}.$$
\end{notation}
Every non-Euclidean disk is also a Euclidean disk:
\begin{equation}\label{eq0}
  K(z_0,r)=D(\frac{1-r^2}{1-r^2|z_0|^2}z_0,r\frac{1-|z_0|^2}{1-r^2|z_0|^2}),
\end{equation}
and in particular, $K(0,r)=D_r$.  We also recall that conformal
automorphisms $\tau$ of $D_1$ are of the form
$\tau(z)=e^{i\theta}\frac{z-z_0}{1-\z_0z}$, where $\tau(z_0)=0$, and
such maps are isometries with respect to $\rho_H$.  More generally, for
any holomorphic map $f:D_1\to D_1$, $\rho_H(f(z),f(w))\le\rho_H(z,w)$.

\begin{prop}\label{prop1.2b}
  Let $f:D_1\to D_1$ be holomorphic, with $f(0)=0$,
  $|f^\prime(0)|=\delta>0$.  Then, for any $\eta$ such that
  $0<\eta<\delta$, \begin{equation}\label{eq14} z\in D_\eta\implies
  |f(z)|>\left(\frac{\delta-\eta}{1-\eta\delta}\right)|z|.
  \end{equation} Further, the restricted function
  $f:D_\eta\to D_1$ takes on each value $w\in
  D_{\left(\frac{\delta-\eta}{1-\eta\delta}\right)\eta}$ exactly once.
\end{prop}
\begin{proof}
  The proof of the first part uses the geometric properties of $\rho_H$.
  Given $f(z)$ holomorphic on $D_1$ with $f(0)=0$, let
  $$h(z)=\left\{\begin{array}{ll} f(z)/z&\mbox{if
    $z\ne0$}\\f^\prime(0)&\mbox{if $z=0$}\end{array}\right..$$ By the
  Schwarz Lemma, $h(z)$ is a holomorphic map $D_1\to D_1$, and
$$\rho_H(h(z),h(0))\le\rho_H(z,0)=|z|,$$ that is,
  $h(z)\in{\overline{K}}(h(0),|z|)$, and by subtracting radius from
  the magnitude of the center in (\ref{eq0}), we get a minimum
  distance from the origin for $|z|\le|h(0)|$:
$$|h(z)|\ge\frac{|h(0)|-|z|}{1-|z||h(0)|}.$$ The RHS is a decreasing
function of $|z|$, so for $0<|z|<\eta<\delta=|h(0)|$,
$$\left|\frac{f(z)}z\right|>\frac{\delta-\eta}{1-\eta\delta}.$$ The
inequality (\ref{eq14}) follows.  In fact, this shows there is a
nonlinear inequality, $$|f(z)|\ge\frac{\delta-|z|}{1-|z|\delta}|z|.$$

  The proof of the second part uses Rouch\'e's Theorem, which we
  recall as follows: Given a contour $\Gamma\subseteq\co$ and two functions
  $f(z)$, $g(z)$ analytic inside and on $\Gamma$, if $|f(z)|>|g(z)|$ at
  each point on $\Gamma$, then $f$ and $f+g$ have the same number of zeros
  inside $\Gamma$ (possibly including multiplicities).

  Starting with the one-to-one claim, let $z_1$ be any point in
  $D_\eta$ such that $f(z_1)\in
  D_{\left(\frac{\delta-\eta}{1-\eta\delta}\right)\eta}$.  Then there
  is some $\epsilon$, $0<\epsilon<\eta$, so that
  $|f(z_1)|<\left(\frac{\delta-\eta}{1-\eta\delta}\right)\epsilon<\left(\frac{\delta-\eta}{1-\eta\delta}\right)\eta$.
  By (\ref{eq14}),
  $\left(\frac{\delta-\eta}{1-\eta\delta}\right)|z_1|<|f(z_1)|<\left(\frac{\delta-\eta}{1-\eta\delta}\right)\epsilon$,
  so $z_1\in D_\epsilon\subsetneq D_\eta$.  Let $\Gamma$ be the
  boundary circle of $D_\epsilon$.  Also by (\ref{eq14}), $f$ has
  exactly one zero inside $\Gamma$, at the origin, with multiplicity
  $1$.  Let $g(z)$ be the constant function $-f(z_1)$.  Then, for any
  $z\in \Gamma$,
  $|f(z)|>\left(\frac{\delta-\eta}{1-\eta\delta}\right)\epsilon$ by
  (\ref{eq14}), so by construction, $f$ and $g$ satisfy the hypotheses
  of Rouch\'e's Theorem.  The conclusion is that
  $f(z)+g(z)=f(z)-f(z_1)$ has exactly one zero in the disk
  $D_\epsilon$, and the claim follows since $\epsilon$ can be
  arbitrarily close to $\eta$.

  The proof of the onto claim is similar.  Let $w$ be any point in
  $D_{\left(\frac{\delta-\eta}{1-\eta\delta}\right)\eta}$, and let
  $\epsilon$ be as above, with
  $|w|<\left(\frac{\delta-\eta}{1-\eta\delta}\right)\epsilon<\left(\frac{\delta-\eta}{1-\eta\delta}\right)\eta$.
  Again letting $\Gamma$ be the boundary of $D_\epsilon$, and letting
  $g(z)$ be the constant function $-w$, we get the same conclusion,
  that $f(z)-w$ has exactly one zero in $D_\epsilon\subseteq D_\eta$.
\end{proof}
In Lemma \ref{lem1.2}, the goal is to get an inverse for $f:D_1\to
D_1$ defined on a large neighborhood of $0$ in the target.  We denote
the radius $\left(\frac{\delta-\eta}{1-\eta\delta}\right)\eta=s$ and
consider for simplicity the condition $s>\frac12$.  The set-up of
(\ref{eq23}) is that $\delta=r/2<1$, and we can pick any $\eta$ in
$(0,r/2)$, so $s>\frac12$ becomes:
$$\left(\frac{r/2-\eta}{1-\eta
r/2}\right)\eta>\frac12\iff0>4\eta^2-3r\eta+2.$$ If the RHS polynomial
has distinct real roots, then there exists an $\eta$ satisfying the
inequality between the two roots.  The roots are given by the
quadratic formula, $\frac18(3r\pm\sqrt{9r^2-32})$.  The condition
$r>4\sqrt{2}/3$ guarantees two roots, and then the midpoint is always
a solution, $\eta=3r/8$, in the interval $(0,r/2)$.  This explains the
choices of constants in Lemma \ref{lem1.2} and its Proof.

\end{document}